\documentclass[a4paper,10pt]{article}

\usepackage[utf8]{inputenc}
\usepackage[english]{babel}
\usepackage{amssymb,amsmath,amsthm}
\usepackage{mathtools}
\usepackage{hyperref}
\usepackage{graphicx}
\usepackage[margin=1in]{geometry}

\newcommand{\ra}{\!\mathop{\rightarrow}}
\newcommand{\sq}{\mathbin{\square}}

\newtheorem{theorem}{Theorem}[section]
\newtheorem{lemma}[theorem]{Lemma}

\newtheorem{conjecture}[theorem]{Conjecture}
\newtheorem{proposition}[theorem]{Proposition}
\newtheorem{question}[theorem]{Question}
\theoremstyle{definition}
\newtheorem{definition}[theorem]{Definition}
\newtheorem{remark}[theorem]{Remark}
\newtheorem{example}[theorem]{Example}
\newtheorem{exercise}[theorem]{Exercise}

\title{Bond percolation does not simulate site percolation}
\author{Nikita~Gladkov \and Alexander~P.~Zimin}
\date{}

\begin{document}
\maketitle

\begin{abstract}
We show that a site percolation is a stronger model than a bond percolation. We use the van den Berg--Kesten (vdBK) inequality to prove that site percolation on a neighborhood of a vertex of degree $4$ cannot be simulated even approximately by bond percolation, and develop a decision tree technique to prove the same for a neighborhood of a vertex of degree $3$. This technique can be used to obtain inequalities for connectedness probabilities, including a conjectured inequality of Erik Aas.
\end{abstract}

% BEGIN inlined content from Text_v2.tex (no external inputs; arXiv flat submission)
\section{Introduction}

For a graph $G=(V,E)$ we consider three independent opening models.
The \emph{(inhomogeneous) Bernoulli bond percolation} $\mu$ opens each
edge $e\in E$ with probability $p_e$, independently of other edges, whereas the
\emph{(inhomogeneous) Bernoulli site percolation} $\sigma$ opens each
vertex $v\in V$ with probability $p_v$.  The parameters $p_e$ and $p_v$
may vary from edge to edge and from vertex to vertex throughout the paper.

Finally, we shall occasionally need a third model.
Given a hypergraph $H=(V,E)$ -- so each $e\in E$ is a finite
subset of $V$ -- the \emph{(inhomogeneous) Bernoulli hyperedge
percolation} $\eta$ opens each hyperedge $e$ with probability
$p_e\in[0,1]$, independently for different $e$.
When $|e|=2$ for every hyperedge, this reduces to the bond model; thus
hyperedge percolation strictly contains bond percolation as a special
case.

\medskip
\noindent\textbf{Observable vertices.}
Throughout the paper we work with a chosen subset $V_{\mathrm{obs}}\subseteq V$
called the \emph{observable vertices}.
We care only about how these vertices are connected to each other.
Vertices in $V\setminus V_{\mathrm{obs}}$, which we call
\emph{auxiliary}, may lie on paths but never serve as endpoints in the
events we consider.
In the bond and hyperedge models, we simply take $V_{\mathrm{obs}}=V$.
In the site model we also set $p_v=1$ for every $v\in V_{\mathrm{obs}}$,
so each observable vertex is always open.

\medskip
\begin{definition}[Connectivity]\label{def:connectivity}
Let $\rho\in\{\mu,\sigma, \eta\}$ be either percolation on $G=(V,E)$ and let
$u,v\in V_{\mathrm{obs}}$.  We write
\[
  u \;\xleftrightarrow[\rho]{}\; v
\]
if there exists a path
$P=(w_0=u,\dots,w_k=v)$ in $G$ such that
\begin{itemize}
\item when $\rho=\mu$ (bond case), every edge $(w_i,w_{i+1})$ of $P$
is open;
\item when $\rho=\sigma$ (site case), every vertex $w_i$ of $P$ is
open, including the endpoints $u$ and~$v$;
\item when $\rho=\eta$ (hyperedge case), for each $i$ there exists an open
hyperedge $e_i\in E$ with ${w_i,w_{i+1}}\subseteq e_i$.
\end{itemize}
We denote the probability of this event by
\[
  \rho(uv)\;:=\;\rho(u\xleftrightarrow[\rho]{}v).
\]
\end{definition}

One can ask many questions about the probabilistic properties of clusters
connected via open vertices and edges.  There are well known
inequalities comparing critical probabilities of site and bond
percolation on the same infinite graph \cite{GriSta}.

To motivate our problem, recall Exercise 3.4 in \cite{Gri} (see also
Exercise 6 in \cite{DC}):
\begin{exercise}\label{Ex:Grim}
  ``Show that bond percolation on a graph $G$ may be reformulated in terms
  of site percolation on a graph derived suitably from $G$.''
\end{exercise}
Here is a formal definition.

\begin{definition}[Exact simulation]\label{def:simulation}
    Let $\rho$ be a percolation measure on a graph $G=(V,E)$ with observable
    set $V_{\mathrm{obs}}$, and let $\rho'$ be a percolation measure on a
    (possibly different) graph $G'=(V',E')$ with observable set
    $V'_{\mathrm{obs}}$.  We say that \emph{$\rho$ simulates $\rho'$} if
    there exists a map $f:V_{\mathrm{obs}}\to V'_{\mathrm{obs}}$ such that,
    for every Boolean combination of the events
    $\{\,v_i\xleftrightarrow[\rho]{}v_j\}$ with $v_i,v_j \in V_{\mathrm{obs}}$,
    \[
      \rho\bigl(\mathcal{E}\bigr)
      \;=\;
      \rho'\bigl(\mathcal{E}\circ f\bigr),
    \]
    where $\mathcal{E}\circ f$ is obtained from $\mathcal{E}$ by replacing
    each $v_i$ with $f(v_i)$ and each connectivity symbol
    $\xleftrightarrow[\rho]{}$ with $\xleftrightarrow[\rho']{}$.
    \end{definition}

\begin{remark}
By this definition, the simulation preserves events such as ``At least $n$ out of $m$ vertices $v_1, \dots, v_m$ are in the same cluster'', but is not guaranteed to preserve the probability of ``There is a path from $a$ to $b$ avoiding vertex $c$''.
\end{remark}

The following theorem \cite{Fis, FE} solves Exercise~\ref{Ex:Grim} by constructing a site percolation that simulates any given
bond percolation.

\begin{theorem}\label{bondinsite}
For every graph $G$ equipped with a bond percolation $\mu$
there exists a graph $G'$ together with a site percolation $\sigma$
that simulates $\mu$.
\end{theorem}

\begin{proof}
    Make a copy $G'$ of $G$ and insert a new auxiliary vertex $w_e$ in the
    middle of each edge $e=\{u,v\}$.  
    We declare all original vertices observable and open with probability $1$,  
    while each auxiliary vertex $w_e$ is declared open independently with
    probability $p_e$.  
    Two original vertices are connected in $\sigma$ iff every $w_e$ on the
    corresponding path is open, i.e.\ exactly when every edge $e$ on that
    path is open in $\mu$.  Thus $\sigma$ simulates $\mu$ in the sense of
    Definition~\ref{def:simulation}.
\end{proof}

Similarly, it is natural to ask whether site percolation can be simulated by bond percolation. 
Fisher \cite{Fis} noted that the other direction cannot be true since the argument proving Theorem~\ref{bondinsite} is only invertible for line graphs. We make his argument precise in Theorem~\ref{sitenotinbond} proved in Section~\ref{remarks}. However, the question becomes more interesting if we consider approximate simulations.

\begin{definition}[Approximate simulation]\label{def:approx_sim}
    Let $\{\rho_i\}$ be a sequence of percolation measures on graphs
    $G_i=(V_i,E_i)$, each with observable set $V_{i,\mathrm{obs}}\subseteq V_i$.
    Let $\rho$ be a percolation measure on a (possibly different)
    graph $G=(V,E)$ with observable set $V_{\mathrm{obs}}\subseteq V$.
    
    We say that \emph{$\{\rho_i\}$ approximately simulates $\rho$} if there
    exist vertex maps
    \[
      f_i : V_{\mathrm{obs}} \longrightarrow V_{i,\mathrm{obs}}
      \qquad(i=1,2,\dots)
    \]
    such that for every Boolean combination $\mathcal{E}$ of connectivity events
    $\{\,v_j\xleftrightarrow[\rho]{}v_k\}$ with $v_j,v_k \in V_{\mathrm{obs}}$,
    \[
      \rho_i\bigl(\mathcal{E}\circ f_i\bigr)
      \;\longrightarrow\;
      \rho\bigl(\mathcal{E}\bigr)
      \qquad\text{as } i\to\infty,
    \]
    where $\mathcal{E}\circ f_i$ is obtained from $\mathcal{E}$ by replacing
    each $v_j$ with $f_i(v_j)$ and each instance of
    $\xleftrightarrow[\rho]{}$ with $\xleftrightarrow[\rho_i]{}$.
\end{definition}

  We denote by \(K_{1,n}\) the star graph with one internal vertex and \(n\) leaves.
  Throughout this section we work with the site percolation model on \(K_{1,n}\) (\(n\ge 3\))
  whose \emph{observable} set is the collection of leaves.
  Every observable leaf is deterministically open (\(p_v=1\)), while the unique
  internal (auxiliary) vertex is open with probability \(p\in(0,1)\).
  Equivalently, this model can be viewed as a \emph{full hyperedge percolation}
  on the same leaf set: opening the internal vertex corresponds to opening the
  single hyperedge \(\{\text{all leaves}\}\) with probability \(p\); when that hyperedge
  is closed (\(1-p\)), no two leaves are connected, whereas when it is open
  all leaves belong to one common cluster.  
  
  The main results of the paper are the following theorems.
  
  \begin{theorem}\label{4vgraph}
  Site percolation on $K_{1,4}$ cannot be approximately simulated by bond
  percolation.
  \end{theorem}
  
  \noindent
  Theorem \ref{4vgraph} is deduced from Theorem~\ref{4hypernotinbond}.
  
  \begin{theorem}\label{3vgraph}
  Site percolation on $K_{1,3}$ cannot be approximately simulated by bond
  percolation.
  \end{theorem}
  
  \noindent
  Theorem \ref{3vgraph} follows from Theorem~\ref{3hypernotinbond}.
  Its proof relies on new inequalities for connectivity events.
  We introduce auxiliary events defined by decision trees to establish
  \eqref{eq:final}, and use computer search to discover further
  inequalities, including \eqref{eq:computer} and~\eqref{e2} in
  Section~\ref{computer}.

  \section{Preliminary remarks}\label{remarks}

 In what follows we mostly work with a hyperedge viewpoint on the site percolation on \(K_{1, n}\) 
 introduced above, since for the full hyperedge percolation one can conveniently assume \(V_{\mathrm{obs}} = V\).  The next result justifies this flexibility.

\begin{theorem}[Site--hyperedge equivalence]\label{thm:site-hyper}
  Every Bernoulli hyperedge percolation measure can be simulated by a
  Bernoulli site percolation measure, and vice versa, in the sense of
  Definition~\ref{def:simulation}.
\end{theorem}

\begin{proof}
\textbf{(i) Hyperedge simulated by site percolation.}
Let \(H=(V,E)\) be a hypergraph equipped with Bernoulli
hyperedge percolation \(\eta\).
Construct a graph \(G'=(V',E')\) and a site percolation
measure \(\sigma\) as follows:
\begin{itemize}
  \item \emph{Vertices:} \(V' := V \cup \{w_e : e\in E\}\),
        adding auxiliary vertex \(w_e\) for each hyperedge \(e\).
  \item \emph{Edges:} for every \(e\in E\) and \(v\in e\) include
        \(\{v,w_e\}\) in \(E'\).
  \item \emph{Site probabilities:}
        \(\sigma(v\text{ open}) = 1\) for \(v\in V\) and
        \(\sigma(w_e\text{ open}) = p_e\).
\end{itemize}
Take observable set \(V_{\mathrm{obs}}\subseteq V\) and set \(V'_{\mathrm{obs}} := V_{\mathrm{obs}}\) together with the identity map \(f:V_{\mathrm{obs}}\to V'_{\mathrm{obs}}\), a \(\sigma\)-open path
\(u-w_{e_1}-\dots-w_{e_m}-v\)
exists if and only if the hyperedges \(e_1,\dots,e_m\) are \(\eta\)-open.
Hence \(\sigma\) simulates \(\eta\).

\medskip
\noindent\textbf{(ii) Site simulated by hyperedge percolation.}
Let \(G=(V,E)\) carry Bernoulli site percolation \(\sigma\) with
\(\sigma(v\text{ open}) = p_v\).
Build a hypergraph \(H'=(V',E')\) and percolation \(\eta'\) as follows:
\[
V' := V \cup \{w_e : e\in E\}, \qquad
E' := \{\,e_v : v\in V\},
\]
\[
e_v := \{v\} \cup \{w_e : v\in e\}, \qquad
\eta'(e_v\text{ open}) = p_v .
\]

Take observable set \(V_{\mathrm{obs}}\subseteq V\) and set \(V'_{\mathrm{obs}}:=V_{\mathrm{obs}}\); let \(g:V_{\mathrm{obs}}\to V'_{\mathrm{obs}}\) be the identity inclusion.

If \(u,v\in V_{\mathrm{obs}}\) are \(\sigma\)-connected, there is a path
\(u = x_0, x_1, \dots, x_k = v\) in \(G\) with all \(x_i\) open.
Each step \(x_i x_{i+1}\) uses the auxiliary vertex \(w_{x_i x_{i+1}}\),
so the open hyperedges \(e_{x_i}\) give an \(\eta'\)-path from \(g(u)\)
to \(g(v)\).

Conversely, any \(\eta'\)-path between \(g(u)\) and \(g(v)\) alternates
\(e_{x_i}\) and \(w_{x_i x_{i+1}}\); hence
\(u = x_0, x_1, \dots, x_k = v\) is a \(\sigma\)-open path in \(G\).
Thus connectivity events coincide, and \(\eta'\) simulates \(\sigma\).
\end{proof}

Note that the hypergraph percolation in the sense of \cite{WZ} is more general than our full hypergraph percolation and is capable of modeling more phenomena.

 Now we know that simulating site percolation is equivalent to simulating full hyperedge percolation. It is easy to see that bond percolation cannot simulate \emph{exactly} even a hyperedge of size $3$ with probability $0 < p < 1$, thus proving Fisher's remark.

Fix a Bernoulli bond percolation $\mu$ on a graph $G=(V,E)$ and write $\mu(\,\cdot\,)$ for probabilities taken with respect to this measure.

\begin{definition}
For pairwise disjoint, non-empty vertex sets $A_1,\dots ,A_k\subseteq V$ we write
$$A_1|A_2|\dots|A_k$$
for the event that
\begin{itemize}
  \item all vertices inside each $A_i$ lie in the \emph{same} open cluster, and
  \item the clusters corresponding to different $A_i$'s are \emph{distinct} (no vertex of $A_i$ is connected to a vertex of $A_j$ for $i\ne j$).
\end{itemize}
\end{definition}

\begin{remark}
Throughout the paper we will use the following shortcuts:
$$
\mu(abc):=\mu(\,\{a,b,c\}\,),\qquad
\mu(ab|c):=\mu(\{a,b\}|\{c\}),\qquad
\mu(a|b|c):=\mu(\{a\}|\{b\}|\{c\}),\ \ldots
$$
Throughout, a vertical bar ``$|$'' indicates that the vertices on either side of the bar must be in different clusters in the configuration and vertices between consecutive bars are in the same connected cluster.
\end{remark}

First, we show that the exact simulation for 3-hyperedges is impossible.

\begin{theorem}\label{sitenotinbond}
Let $G=(V,E)$ be any graph, let $\mu$ be a Bernoulli bond percolation on $G$, and let $a,b,c\in V$.  Then either
$$\mu(abc)+\mu(a|b|c)<1\qquad\text{or}\qquad \max\{\mu(abc),\,\mu(a|b|c)\}=1.$$
\end{theorem}
\begin{proof}
Delete every edge whose parameter is~$0$ and contract every edge whose parameter is~$1$; all remaining edges now have parameters strictly between $0$ and~$1$.

Assume for a contradiction that
\begin{equation}\label{eq:contradiction}
0<\mu(abc),\ \mu(a|b|c)<1
\quad\text{and}\quad
\mu(abc)+\mu(a|b|c)=1 .
\end{equation}

Because $\mu(a|b|c)>0$, each of the mixed-cluster events
\[
ab|c,\quad ac|b,\quad a|bc
\]
must have probability $0$, otherwise the left-hand side of \eqref{eq:contradiction} would be strictly smaller than~$1$.

Let us show that $\mu(ab|c)=0$ forces \emph{every} $a$--$b$ path in $G$ to pass through $c$.   
Indeed, if a path $\gamma$ joining $a$ to $b$ while avoiding $c$ existed in $G$, the configuration in which the edges of $\gamma$ are declared open and all other edges are declared closed would realise the event $ab|c$ with positive probability (all edge parameters lie in $(0,1)$), contradicting $\mu(ab|c)=0$.  Repeating the same argument for the other two zero-probability events we deduce

\begin{itemize}
\item every $a$--$b$ path visits $c$,
\item every $a$--$c$ path visits $b$, and
\item every $b$--$c$ path visits $a$.
\end{itemize}

These three properties cannot be satisfied simultaneously.  Indeed, the first property forces any $a$--$b$ path to begin with an $a$--$c$ sub-path; by the second property that sub-path must pass through $b$ before it reaches $c$.  Hence \emph{any} $a$--$b$ path would have to contain vertex $b$ twice, which is impossible for a (simple) path in a graph.  Thus these properties are impossible, contradicting \eqref{eq:contradiction}.

Consequently our assumption is false and the theorem follows.
\end{proof}

The theorem implies that bond percolation cannot simulate a non-trivial hyperedge percolation on the graph containing exactly three vertices and a single hyperedge connecting these. Indeed, in such a hyperedge percolation, the probability that all three vertices are connected equals $p$ (the probability that the hyperedge is open), while the probability that they are all disconnected equals $1-p$. Theorem~\ref{sitenotinbond} shows that no bond percolation can achieve both these probabilities simultaneously.

Although Theorem \ref{sitenotinbond} prohibits exact simulation of hyperedge percolation by bond percolation, we consider whether it is possible to have an arbitrarily good approximation.

\begin{question}\label{question_k}
For given $k$, $p \not\in \{0, 1\}$ and $\varepsilon>0$, does there exist a graph $G = (V, E)$, containing vertices $x_1$, \dots, $x_k$ and a bond percolation $\mu$ on it with $\mu(x_1x_2\dots x_k) > p-\varepsilon$ and $\mu(x_1|x_2|\dots|x_k) > 1-p-\varepsilon$?
\end{question}

In Section~\ref{4v} we show that approximate simulation is impossible for $k \ge 4$ using a lemma due to Hutchcroft \cite{Hut}, thus proving Theorem~\ref{4vgraph}.
Finally, we develop a new technique using decision trees to resolve Question \ref{question_k} for $k=3$ (thereby proving Theorem~\ref{3vgraph}) in Section~\ref{3v}.

\section{Simulating \boldmath\texorpdfstring{$k$}{k}-hyperedge for \boldmath\texorpdfstring{$k \ge 4$}{k>=4}}\label{4v}
In \cite{Hut}, the following theorem is proved using the vdBK inequality~\cite{BK}, where $K_u$ is the cluster containing vertex $u$, and for each finite subset $\Lambda \subseteq V$
$$|K_{max}(\Lambda)| = \max\{|K_v \cap \Lambda| : v \in V\}$$ 
is the maximal number of vertices from $\Lambda$ belonging to the same cluster.

\begin{theorem}[\cite{Hut}, Theorem 2.3]
Let $G = (V, E)$ be a countable graph and let $\Lambda \subseteq V$ be finite and non-empty. Then for any Bernoulli bond percolation $\mu$ on $G$ one has
\begin{equation}\label{hc}
\mu(|K_{max}(\Lambda)| \ge 3^k\lambda) \le \mu(|K_{max}(\Lambda)| \ge \lambda)^{3^{k-1}+1}
\end{equation}
and 

\begin{equation}\label{hc+}
\mu(|K_u\cap \Lambda| \ge 3^k\lambda) \le \mu(|K_{max}(\Lambda)| \ge \lambda)^{3^{k-1}}\mu(|K_u\cap \Lambda| \ge \lambda)
\end{equation}
for every $\lambda \ge 1$ (not necessarily integer), integer $k \ge 0$ and $u \in V$.
\end{theorem}

This allows us to prove that one cannot even approximately simulate the $4$-hyperedge.

\begin{theorem}\label{4hypernotinbond}
For any $\delta>0$ there exists an $\varepsilon>0$ such that, for every graph $G=(V,E)$, every Bernoulli bond percolation $\mu$ on $G$, and every choice of four vertices $a,b,c,d\in V$, at least one of the following holds:
\begin{enumerate}
  \item $\mu(abcd)\;\ge 1-\delta$;
  \item $\mu(a|b|c|d)\;\ge 1-\delta$;
  \item $\mu(abcd)+\mu(a|b|c|d)\;< 1-\varepsilon$.
\end{enumerate}
\end{theorem}
\begin{proof}
    Let \(\delta>0\) and set \(\varepsilon=\delta^2/2\).  Suppose for some graph \(G\), bond percolation \(\mu\), and vertices \(a,b,c,d\) we have
    \[
    \mu(abcd)\le1-\delta
    \quad\text{and}\quad
    \mu(a|b|c|d)\le1-\delta.
    \]
    Write \(t=\mu(a|b|c|d)\in[0,1-\delta]\).  Then~\eqref{hc} with \(\lambda=\tfrac43\) and \(\Lambda = \{a, b, c, d\}\) gives
    \begin{equation}\label{eq:4hyperineq}
        \mu(abcd)\le\mu(ab\cup ac\cup ad\cup bc\cup bd\cup cd)^2.
    \end{equation}
    Since \(ab\cup ac\cup ad\cup bc\cup bd\cup cd\) is the complement of \(a|b|c|d\), its probability is \(1-t\), so
    \[
    \mu(abcd)\le(1-t)^2.
    \]
    Hence
    \[
    \mu(abcd)+\mu(a|b|c|d)=\mu(abcd)+t\le\min\{\,1-\delta+t,\;t+(1-t)^2\}.
    \]
    If \(t\le\delta/2\) then \(1-\delta+t\le1-\delta/2\le1-\varepsilon\), while if \(t\ge\delta/2\) then
    \[
    t+(1-t)^2
    =1-t(1-t)
    \le1-\tfrac{\delta}{2}\,\delta
    =1-\varepsilon.
    \]
    In either case \(\mu(abcd)+\mu(a|b|c|d)\le1-\varepsilon\), which is alternative (iii).  
    \end{proof}

\begin{remark}
    Another proof of \eqref{eq:4hyperineq} can be obtained by applying the vdBK inequality directly. For two events $A$, $B$, their disjoint occurrence $A \sq B$ is defined as the event consisting of configurations $x$ whose memberships in $A$ and in $B$ can be verified on disjoint subsets of edges. In this case, vdBK inequality asserts that
    \[
    \mu(A\sq B) \le \mu(A)\mu(B).
    \]
    
Let \(E\) be the event \(ab \cup ac \cup ad \cup bc \cup bd \cup cd\). Then \(E \sq E\) represents the event where there exist two edge-disjoint paths between the vertices \(a\), \(b\), \(c\), and \(d\). Hutchcroft's argument in this context reduces to demonstrating that \(abcd \subseteq E \sq E\), which is established in the following lemma.

\begin{lemma}
Let \(G\) be a connected graph, and let \(a_1, a_2, a_3, a_4\) be any four vertices in \(G\). Then, it is always possible to partition these vertices into two disjoint pairs, \(\{a_{i_1}, a_{i_2}\}\) and \(\{a_{j_1}, a_{j_2}\}\), such that there exist edge-disjoint paths \(\gamma_1\) and \(\gamma_2\) in \(G\), where \(\gamma_1\) connects \(a_{i_1}\) to \(a_{i_2}\) and \(\gamma_2\) connects \(a_{j_1}\) to \(a_{j_2}\).
\end{lemma}
\begin{proof}
Partition the vertices \(a_1, a_2, a_3, a_4\) into two disjoint pairs, \(\{a_{i_1}, a_{i_2}\}\) and \(\{a_{j_1}, a_{j_2}\}\), such that the sum of edge distances, \(d(a_{i_1}, a_{i_2}) + d(a_{j_1}, a_{j_2})\), is minimized across all possible partitions. Without loss of generality, assume that the vertices in the first pair are \(a_1\) and \(a_2\), and the vertices in the second pair are \(a_3\) and \(a_4\).

Let \(\gamma_1\) and \(\gamma_2\) be the shortest paths connecting \(a_1\) to \(a_2\) and \(a_3\) to \(a_4\), respectively. Suppose the edge \(uv\) belongs to both paths. Without loss of generality, we assume that the order of vertices in \(\gamma_1\) is \(a_1 \to u \to v \to a_2\); otherwise, we can swap the labels of \(a_1\) and \(a_2\). Similarly, we assume that the order of vertices in \(\gamma_2\) is \(a_3 \to u \to v \to a_4\).

Then, we calculate the following:
\begin{align*}
d(a_1, a_3) + d(a_2, a_4) &\leq \big(d(a_1, u) + d(u, a_3)\big) + \big(d(a_2, v) + d(v, a_4)\big) \\
&= \big(d(a_1, u) + d(v, a_2)\big) + \big(d(a_3, u) + d(v, a_4)\big) \\
&= \big(d(a_1, a_2) - 1\big) + \big(d(a_3, a_4) - 1\big).
\end{align*}

This inequality contradicts the selection of the pairs, as the pairs \(\{a_1, a_2\}\) and \(\{a_3, a_4\}\) were chosen to minimize \(d(a_1, a_2) + d(a_3, a_4)\). Therefore, the paths \(\gamma_1\) and \(\gamma_2\) must be edge-disjoint.
\end{proof}
% Consider a `` centroid '' vertex $v$ of the tree such that the sum of the distances from $v$ to $a, b, c$ and $d$ is minimal. It is an exercise in graph theory to show that there are two path between $a, b, c, d$ in the tree that share no vertices or share only $v$ as a common vertex.
\end{remark}

\begin{proof}[Proof of Theorem~\ref{4vgraph}]
    Fix an arbitrary \(p\in(0,1)\) and let \(\sigma_p\) be the site
    percolation on the star \(K_{1,4}\) in which the centre is open with
    probability \(p\) while the four leaves \(a,b,c,d\) are always open.
    Then
    \[
    \sigma_p(abcd)=p,\qquad
    \sigma_p(a|b|c|d)=1-p,\qquad
    \sigma_p(abcd)+\sigma_p(a|b|c|d)=1 .
    \]
    
    Choose \(\delta:=\frac12\min\{p,1-p\}>0\) and let
    \(\varepsilon>0\) be the constant provided by
    Theorem~\ref{4hypernotinbond}.  
    Assume, for contradiction, that \(\sigma_p\) can be approximately
    simulated by some bond-percolation measure \(\mu\).
    Make the simulation so accurate that
    \[
    |\mu(abcd)-p|<\delta,\qquad
    |\mu(a|b|c|d)-(1-p)|<\delta,\qquad
    |\mu(abcd)+\mu(a|b|c|d)-1|<\varepsilon/2 .
    \]
    
    Because each of \(\mu(abcd)\) and \(\mu(a|b|c|d)\) lies in
    \((\delta,1-\delta)\), conditions~(1) and~(2) of
    Theorem~\ref{4hypernotinbond} fail.
    The third bound yields
    \(\mu(abcd)+\mu(a|b|c|d) > 1-\varepsilon\),
    so condition~(3) fails as well.
    This contradiction shows that no bond percolation can approximately
    simulate \(\sigma_p\).
    \end{proof}

\section{Simulating 3-hyperedge: human proof}\label{3v}

%Now we see that it is impossible to even approximately simulate site percolation on $K_{1, 4}$ with bond percolation, as promised in Theorem~\ref{4vgraph}. To prove Theorem~\ref{3vgraph}, we need the following lemma.
In the previous section, we showed that a $4$-hyperedge cannot be approximately simulated by bond percolation; the key ingredient was the BK inequality. For $k=3$, however, the BK inequality appears too weak to rule out such a simulation. This case is of independent interest. Hollom~\cite{Hollom} recently constructed counterexamples to the hyperedge- and site-percolation versions of the bunkbed conjecture. If approximate simulation of a $3$-hyperedge by bond percolation were possible, then Hollom's construction could be transferred to the bond setting, which is the original formulation of the conjecture due to Kasteleyn. Notably, Hollom's hyperedge counterexample uses only $3$-hyperedges, which naturally motivates the question of whether a $3$-hyperedge can, in fact, be simulated by bond percolation. Ultimately, this line of reasoning led in~\cite{GPZ} to the simulation of a weaker WZ hyperedge percolation~\cite{WZ}; although weaker than full hyperedge percolation, it nevertheless suffices to produce a bond-percolation counterexample to the bunkbed conjecture.
 
\begin{definition}
For two configurations $C_1, C_2 \in \Omega = 2^{[E]}$ and a set $S \subseteq E$ we denote by $C_1 \ra_S C_2$ % $C_1|_S \cup C_2|_{\bar{S}}$
the configuration which coincides with $C_1$ on $S$ and $C_2$ on its complement $\overline S$.
\end{definition}

\begin{lemma}\label{tree_lemma}
Consider two independent Bernoulli bond percolations $C_1$ and $C_2$ having the same distribution $\mu$ on the same graph $G$. Let a decision tree $T$ select each edge and reveal it in both $C_1$ and $C_2$. Furthermore, allow on each step, before revealing, to decide if this edge will go to the set $S$ (thus dependent on $C_1$ and $C_2$) or to its complement $\overline S$. Then $C_1 \ra_S C_2$ is independent of $C_1 \ra_{\bar S} C_2$ and both of them are distributed as $\mu$.
\end{lemma}

\begin{example}\label{ex:2edges}
If the graph is a path of length $2$ from $a$ to $b$, then the tree $T$ in Figure~\ref{fig:tree} builds a set $S$ of all edges with one end in the component of $a$ in $C_1$.

\begin{figure}[!hbt]
\centering
\includegraphics[width=1\textwidth]{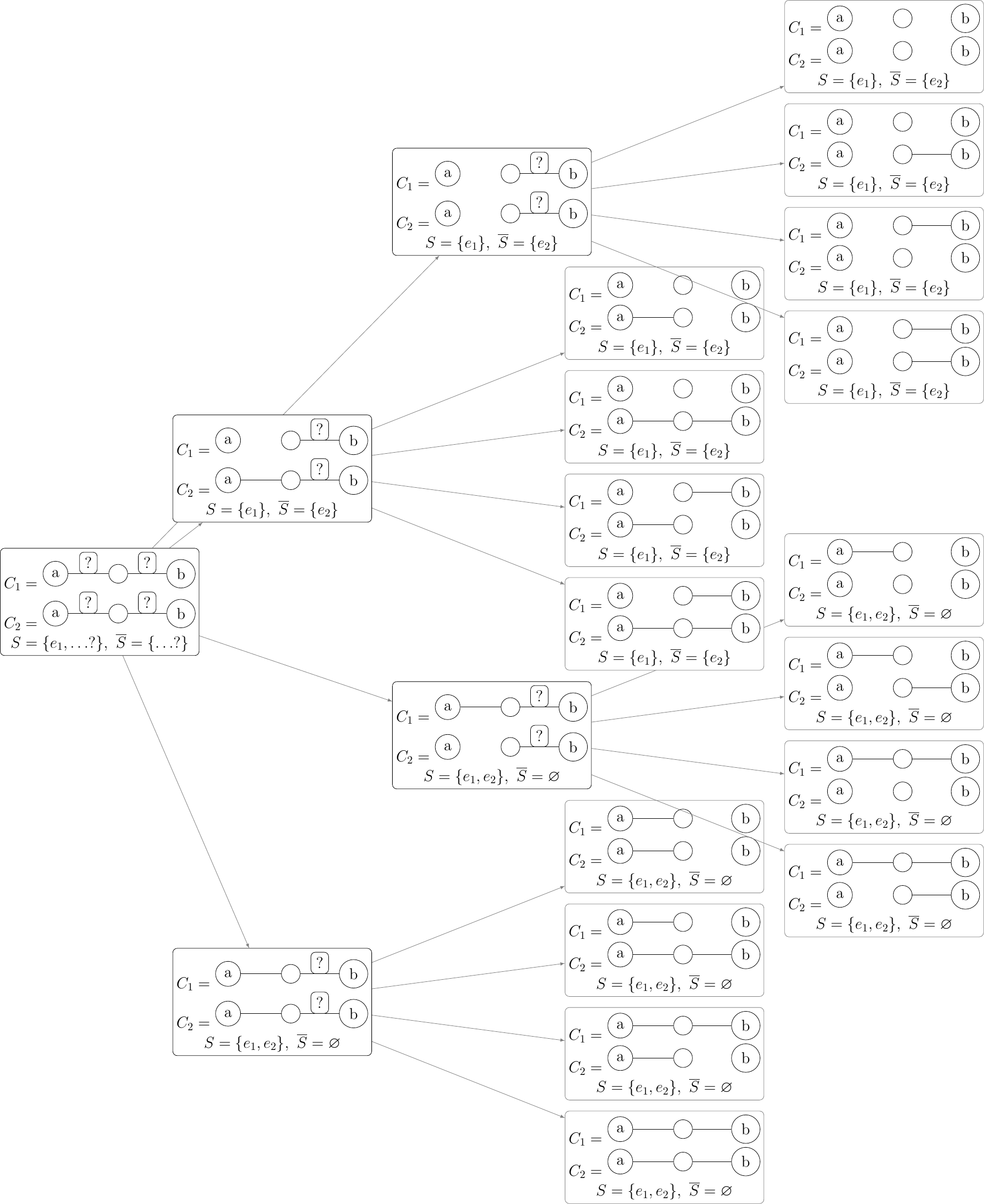}

   \caption{$T$ corresponding to the Example~\ref{ex:2edges}}\label{fig:tree}
\end{figure}
\end{example}

\begin{proof}[Proof of Lemma~\ref{tree_lemma}]
For every pair of configurations $C_3$, $C_4$ and given decision tree $T$ there exist unique $C_1$ and $C_2$ such that $C_1 \ra_{S(C_1, C_2)} C_2 = C_3$ and $C_1 \ra_{\bar S(C_1, C_2)} C_2 = C_4$. Indeed,  the path in $T$ leading to $C_1 \ra_{S(C_1, C_2)} C_2 = C_3$ and $C_1 \ra_{\bar S(C_1, C_2)} C_2 = C_4$ is determined uniquely at each step, and the probability of this path is equal to $\mu(C_3)\mu(C_4)$, which is equal to $\mu(C_1)\mu(C_2)$ since the probability in Bernoulli percolation is a product of probabilities for individual edges.
\end{proof}

\begin{example}\label{SConnected}
One can take $T$ querying all the edges from the vertices already known to connect to the vertex $a$ in $C_1$. It will assign all these edges to $S$ and then discover the remaining edges, assigning them to $\bar{S}$. Then $S$ will be the set of all open and closed edges with at least one edge in the component of $a$. 

Note that this set $S$ depends only on $C_1$. Given this, the configuration $C_1 \ra_S C_2$ can be interpreted as follows. We take the configuration $C_1$ and resample all the edges not connected to the cluster of $a$. Lemma~\ref{tree_lemma} claims that the resulting configuration has a distribution $\mu$. Moreover, if instead we resampled the edges connected to the cluster of $a$, it would also result in measure $\mu$. 
\end{example}

\begin{remark}
Notice that the Markov chain method from \cite{BHK} is based on the fact that resampling edges in $\overline{S}$ from Example~\ref{SConnected} preserves the measure restriction $\mu|_{a|b}$.
In our notation, it means that for $A = a|b$ and any $B$, one has
\begin{equation}\label{Markov}
\mu(C_1 \in A\text{ and } C_1 \ra_{S} C_2 \in B) = \mu(C_1 \ra_{S} C_2 \in A \cap B) = \mu(A \cap B)
\end{equation}

\end{remark}

\begin{theorem}\label{3hypernotinbond}
    For any \(\delta>0\) there exists an \(\varepsilon>0\) such that, for every graph \(G=(V,E)\), every Bernoulli bond percolation \(\mu\) on \(G\), and every choice of three vertices \(a,b,c\in V\), at least one of the following holds:
    \begin{enumerate}
      \item \(\displaystyle \mu(abc)\;\ge\;1-\delta\);
      \item \(\displaystyle \mu\bigl(a|b|c\bigr)\;\ge\;1-\delta\);
      \item \(\displaystyle \mu(abc)+\mu\bigl(a|b|c\bigr)\;<\;1-\varepsilon\).
    \end{enumerate}
\end{theorem}
    
\begin{remark}
It is worth noting that Theorem~\ref {3hypernotinbond} directly implies Theorem~\ref{4hypernotinbond}. 
\end{remark}
\begin{proof}

We will need multiple sets $S_i$ for our purpose. So, we define sets $S_1$, $S_2$ and $S_3$, which are somewhat complex (See Figure~\ref{fig:sets}).

\begin{figure}[tb]
\centering
\begin{minipage}{0.8\textwidth}
\begin{center}
\includegraphics[width=0.3\linewidth]{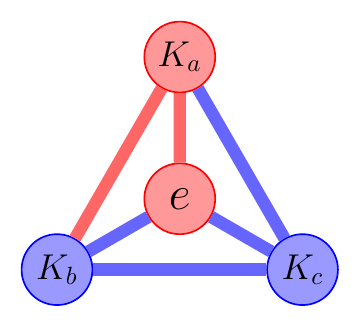}%
\includegraphics[width=0.3\linewidth]{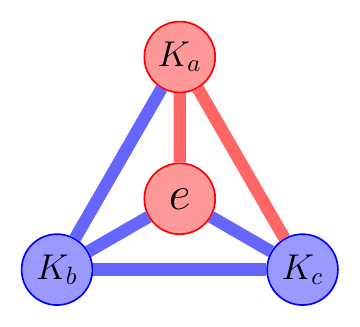}%
\includegraphics[width=0.3\linewidth]{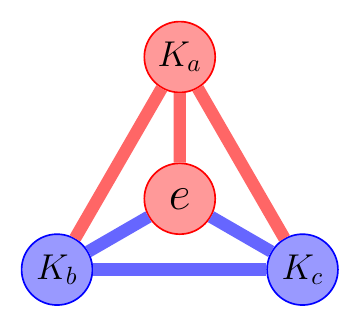}
\caption{$S_1$, $S_2$ and $S_3$ for the case $C_1 \in a|b|c$. Regions surrounding $a, b, c$ depict $K_a$, $K_b$ and $K_c$. Respective sets are in blue and their complements are in red.}\label{fig:sets}
\end{center}

\end{minipage}
\end{figure}

To build $S_1$, we query all edges connected to $c$ and put them in $S$. Then we query all not queried edges connected to $a$ (this is vacuous if $a$ was connected to $c$) and put them in $\bar{S}$. Then we query all not queried edges connected to $b$ and put them in $S$. Finally, we put the rest of the edges in $\bar{S}$. We denote by $K_x$ the set of vertices connected to $x$ via edges open in $C_1$. Then,

$$S_1 = 
\begin{cases}E \cap \big(K_c \times V \cup K_b \times \overline{K_a}\big) \text{ if $C_1 \in a|b|c$;}\\
E \cap \big(K_c \times V\big) \text{ if $C_1$ is in $abc$, $a|bc$ or $ab|c$;} \\
E \cap \big((K_b \cup K_c) \times V\big)\text{ if $C_1 \in ac|b$.}
\end{cases}$$

The only case we will actually use is $a|b|c$. $S_2$ is defined analogously with $b$ and $c$ interchanged.

$$S_2 = 
\begin{cases}E \cap \big(K_b \times V \cup K_c \times \overline{K_a}\big) \text{ if $C_1 \in a|b|c$;}\\
E \cap \big(K_b \times V\big) \text{ if $C_1$ is in $abc$, $a|bc$ or $ac|b$;} \\
E \cap \big((K_b \cup K_c) \times V\big)\text{ if $C_1 \in ab|c$.}
\end{cases}$$

Finally, for $S_3$ we put all edges connected to $a$ in $\bar{S}$, all not queried edges connected to $b$ or $c$ to $S$ and the rest of the edges to $\bar{S}$.

$$S_3 = \begin{cases}E \cap \big((K_b \cup K_c) \times \overline{K_a}\big)\text{ if $C_1 \in a|b|c$ or $a|bc$;}\\
\varnothing \text{ if $C_1 \in abc$;}\\
\text{Something else otherwise.}
\end{cases}$$

The key observation is the following:

\begin{proposition}\label{prop:key}
    For any two configurations $C_1$ and $C_2$ such that $C_1 \in a|b|c$ and $C_1\ra_{S_3}C_2 \in ab \cup ac$, one has either $C_1\ra_{S_1}C_2 \in ab$ or $C_1\ra_{S_2}C_2 \in ac$.
\end{proposition}

\begin{proof}
    Consider a path $\mathcal{P}$ from $a$ to $b$ or $c$ in $C_1\ra_{S_3}C_2$. Along this path, consider the first edge $uv$ incident to a vertex $v$ in $K_b \cup K_c$. Denote by $e$ the set of all the vertices of $G$ which do not belong to $K_a$, $K_b$ or $K_c$. The segment of $\mathcal{P}$ before $uv$ is contained in the complement of $K_b \cup K_c$, which is $K_a \cup e$, and thus it is contained in all of the sets $\bar{S_1}$, $\bar{S_2}$, and $\bar{S_3}$ (see Figure~\ref{fig:sets}: the sets $K_a$ and $e$ are shown in red, along with all the edges between them, in all the sets $S_i$). 
    
    The vertex $u$ belongs to $K_a \cup e$. We now show that $u$ must be in $K_a$. Indeed, suppose it is not; then the edge $uv$ connects $e$ with $K_b \cup K_c$. In $C_1\ra_{S_3}C_2$, this edge comes from $C_1$ (on Figure~\ref{fig:sets}, all such edges are blue). However, in $C_1$ the edges between $e$ and $K_b \cup K_c$ are closed, which would imply $uv$ is closed in $C_1\ra_{S_3}C_2$, a contradiction since $uv$ lies on the open path $\mathcal{P}$.

    Therefore, $u \in K_a$ and $v \in K_b \cup K_c$. Depending on whether $v \in K_b$ or $v \in K_c$, the edge $uv$ belongs to $\bar{S_1}$ or $\bar{S_2}$. Since all internal edges in $K_b$ and $K_c$ belong to $S_1$ and $S_2$, it follows that $C_1\ra_{S_1}C_2 \in ab$ or $C_1\ra_{S_2}C_2 \in ac$.
\end{proof}

The consequence of the Proposition~\ref{prop:key} is the following inequality: 

\begin{multline} \label{eq:prop_cons}
\mu\Big(C_1 \in a|b|c \text{ and }C_1 \ra_{S_3} C_2 \in (ab \cup ac)\Big) \\ \le \mu(C_1 \in a|b|c \text{ and } C_1 \ra_{S_1} C_2 \in ab) + \mu(C_1 \in a|b|c \text{ and } C_1 \ra_{S_2} C_2 \in ac).
\end{multline}

Let's proceed to estimate the probabilities of these events. For $C_1 \in a|b|c$, we have $C_1 \ra_{S_1} C_2 \in a|c$, so
$$\mu(C_1 \in a|b|c \text{ and } C_1 \ra_{S_1} C_2 \in ab) \le \mu(C_1 \ra_{S_1} C_2 \in ab|c) = \mu(ab|c).$$
Similarly, 
$$\mu(C_1 \in a|b|c \text{ and } C_1 \ra_{S_2} C_2 \in ac) \le \mu(ac|b).$$

Finally, we estimate $\mu\left(C_1 \in a|b|c \text{ and } C_1 \ra_{S_3} C_2 \in (ab \cup ac)\right)$.
If $C_1$ belongs to $a|b \cap a|c$, then $\bar{S_3}$ contains a cut from $a$ to $b$ and $c$, so $C_1 \ra_{\bar{S_3}} C_2$ also belongs to $a|b \cap a|c$.
\begin{align*}
\mu\Big(C_1 \in a|b|c &\text{ and } C_1 \ra_{S_3} C_2 \in (ab \cup ac)\Big) \\
&\ge \mu\Big(C_1 \in (a|b \cap a|c) \text{ and } C_1 \ra_{S_3} C_2 \in (ab \cup ac)\Big) - \mu(a|bc) \\
&= \mu\Big(C_1 \ra_{\overline{S_3}} C_2 \in (a|b \cap a|c) \text{ and } C_1 \ra_{S_3} C_2 \in (ab \cup ac)\Big) - \mu(a|bc) \\
&= \mu(a|b \cap a|c)\mu(ab \cup ac) - \mu(a|bc).
\end{align*}

Substituting our bounds into \eqref{eq:prop_cons}, we conclude

\begin{equation}\label{eq:final}
\mu(a|b \cap a|c)\mu(ab \cup ac) \le \mu(ab|c) + \mu(ac|b) + \mu(a|bc).
\end{equation}

To conclude the proof of Theorem~\ref{3hypernotinbond}, assume that alternatives (i) and (ii) both fail, and verify that alternative (iii) holds. Let \(\delta>0\) and put
\[
\varepsilon \;=\;\frac{\delta^2}{4}.
\]
Suppose for some graph \(G\), Bernoulli percolation \(\mu\), and vertices \(a,b,c\) we have
\[
\mu(abc)\le1-\delta
\quad\text{and}\quad
\mu(a|b|c)\le1-\delta.
\]
Since trivially \(\mu(ab\cup ac)\ge\mu(abc)\) and \(\mu(a|b\cap a|c)\ge\mu(a|b|c)\), the displayed inequality gives
\[
\mu(abc)\,\mu(a|b|c)
\;\le\;
\mu(ab|c)+\mu(ac|b)+\mu(a|bc)
\;=\;
1-\mu(abc)-\mu(a|b|c).
\]
Hence
\[
\mu(abc)+\mu(a|b|c)
\;\le\;
1-\mu(abc)\,\mu(a|b|c).
\]
If \(\mu(abc)\le\delta/2\) (or similarly \(\mu(a|b|c)\le\delta/2\)), then
\[
\mu(abc)+\mu(a|b|c)
\le\frac\delta2+(1-\delta)
=1-\frac\delta2
<1-\frac{\delta^2}{4}
=1-\varepsilon.
\]
Otherwise \(\mu(abc),\mu(a|b|c)\ge\delta/2\) and so
\[
\mu(abc)+\mu(a|b|c)
\le1-\frac\delta2\cdot\frac\delta2
=1-\frac{\delta^2}{4}
=1-\varepsilon.
\]
In either case
\(\mu(abc)+\mu(a|b|c)\le1-\varepsilon\),
which is precisely alternative (iii).  This completes the proof.
\end{proof}

    \begin{remark}
    Theorem~\ref{3vgraph} follows from here by considering site percolation on $K_{1, 3}$ in the same manner as Theorem~\ref{4vgraph} was deduced from Theorem~\ref{4hypernotinbond}.
    \end{remark}

From the equation~\eqref{eq:final}, one can conclude that if $\mu(abc)$ and $\mu(a|b|c)$ are simultaneously greater than $p$, then \(p(1-p) \le 1-2p\) and so \(p \le \frac{3-\sqrt{5}}{2} \approx 0.382\). If we denote the maximal possible value of $\min\big(\mu(abc), \mu(a|b|c)\big)$ for any bond percolation by $\alpha_3$, we get an estimate $\alpha_3 < 0.382$, which we improve in the next section. The lower bound $\alpha_3 > 0.29065$ is given in Appendix~\ref{alpha3}.

\section{Simulating 3-hyperedge: computer-assisted proof} 
\label{computer}

Consider a bond percolation on a graph $G$ with specified vertices $a$, $b$, and $c$. We examine a set of decision trees $T_k$ for $k = 1, \dots, n$. These trees generate the sets $S_k$ and $\overline{S}_k$. Each tree maps every pair of configurations $C_1$ and $C_2$ on $G$ into a product space $J^2$, with $J = \{a|b|c, a|bc, ac|b, ab|c, abc\}$. The first coordinate maps the partition of vertices $a, b, c$ in the graph $C_1 \ra_{S_k} C_2$ into connected clusters, and the second coordinate corresponds to the partition of vertices in the graph $C_1 \ra_{\overline{S}_k} C_2$.\footnote{Take into account that there is a trivial tree $T$ that generates the set of all the edges $S = E$; for this tree, $C_1 \ra_S C_2 = C_1$ and $C_1 \ra_{\overline{S}} C_2 = C_2$.}

The bond percolation on $G$ induces a joint probability distribution $\rho$ on $(J^2)^n$. However, it is important to note that $\rho$ may be restricted in several ways. Not every combination of partitions $(p_k, \bar{p}_k)$ for $k = 1, \dots, n$ corresponds to a pair of actual configurations $C_1$ and $C_2$ that satisfy conditions $C_1 \ra_{S_k} C_2 = p_k$ and $C_1 \ra_{\overline{S}_k} C_2 = \bar{p}_k$. Denote by $F \subseteq (J^2)^n$ the set of all feasible combinations of these partitions. Thus, the support of $\rho$ is contained within $F$.

Let $\mu$ be a probability distribution on $J$ induced by bond percolation on $G$. Lemma~\ref{tree_lemma} implies that all the marginal projections of $\rho$ onto each $(J^2)_k$ are identical and equal to $\mu \times \mu$. Given these restrictions, we can formulate a necessary condition for the implementability of the distribution $\mu$ on $J$, expressed through the feasibility of a linear programming problem.

\begin{proposition}\label{prop:primal_criterion}
    For any feasible distribution $\mu$ on $J$, there exists a distribution $\rho(p) \ge 0$ defined on the set $p \in F$, satisfying the following condition:
    \begin{align}\label{eq:primal_constraint}
    \sum_{\substack{p \in F\colon \\ p_k = q,\,\overline{p}_k = \overline{q}}} \rho(p) = \mu(q) \cdot \mu(\bar{q}) \quad &\text{for all $k = 1, \dots, n$, and for each pair $(q, \overline{q}) \in J^2$}.
    \end{align}
\end{proposition}

The above proposition does not offer a necessary condition that can be expressed solely in terms of $\mu$ without the introduction of additional variables. To achieve a formulation that depends only on $\mu$, we consider the dual linear programming problem to the one considered in Proposition~\ref{prop:primal_criterion}.

\begin{definition}
    \textit{Feasible potentials} are the collection of functions $\varphi_k\colon J^2 \to \mathbb{R}$ satisfying the inequalities
    \begin{equation}\label{eq:dual_feasibility}
    \sum_{k = 1}^n \varphi_k(p_k, \overline{p}_k) \ge 0 \quad \text{for all feasible $p \in F$}.
    \end{equation}
    Each function $\varphi_k$ can be interpreted as a variable dual to the marginal projection constraint \eqref{eq:primal_constraint}.
\end{definition}

Next, we utilize the principle that the feasibility of the primal linear program is equivalent to the boundedness of the dual linear program.

\begin{theorem}\label{thm:dual_potentials}
    Let $\varphi_k$ be feasible potentials. Then, any feasible distribution $\mu$ on $J$ satisfies the inequality
    \begin{equation}
    \sum_{k = 1}^n \sum_{(p, \overline{p}) \in J^2} \varphi_k(p, \overline{p}) \mu(p) \mu(\overline{p}) \ge 0.
    \end{equation}
\end{theorem}

\begin{proof}
    Let $\mu$ be a feasible distribution on $J$. By Proposition~\ref{prop:primal_criterion}, we can find a joint law $\rho$ supported on the set $F$ of feasible combinations of partitions with all marginal projections equal to $\mu \times \mu$. The latter condition implies that
    \[
    \sum_{k = 1}^n \sum_{(p, \overline{p}) \in J^2} \varphi_k(p, \overline{p}) \mu(p) \mu(\overline{p}) = \sum_{k = 1}^n \sum_{p \in F} \varphi_k(p_k, \overline{p}_k) \rho(p).
    \]
    By the definition of feasible potentials, the inequality
    \[
    \sum_{k = 1}^n \varphi_k(p_k, \overline{p}_k) \ge 0
    \]
    holds for all $p \in F$; thus, the right-hand side of the equation above is non-negative.
\end{proof}

This theorem allows to prove the inequality
\begin{equation}\label{eq:computer}
\mu(a|b \cap a|c)\mu(ab \cup ac) \le \mu(ab|c) + \mu(ac|b) + \mu(a|bc) - \mu(ab|c)^2-\mu(ac|b)^2,
\end{equation}
which is obviously better than inequality~\eqref{eq:final} and leads to an estimate $\alpha_3 \le 0.369$. 
Moreover, surprisingly, this theorem also proves the inequality 
\begin{equation}\label{e2}
\mu(abc)\mu(a|b|c) \ge \mu(ab|c)\mu(ac|b)+\mu(ab|c)\mu(a|bc)+\mu(ac|b)\mu(a|bc),
\end{equation}
which was first conjectured in an unpublished work by Erik Aas and proved in \cite{G}. It is stronger than what the Harris--Kleitman inequality can tell about these events. 

To prove the inequality \eqref{eq:computer}, we formulate it in terms of the feasible distribution $\mu$.

\begin{proposition}
    Any feasible distribution $\mu$ on $J$ satisfies the inequality 
    \[
\mu(ab \cup ac) \cdot \mu(a|b \cap a|c) + \mu(ac|b)^2 + \mu(ab|c)^2 \le \mu(a|bc) + \mu(ac|b) + \mu(ab|c).
    \]
\end{proposition}
\begin{proof}
   Consider the decision trees constructing the sets $S_1$, $S_2$, $S_3$, and their complements as introduced in the proof of Theorem~\ref{3hypernotinbond}. In addition, we include a trivial decision tree that constructs the sets $S_0 = E$ and $\overline{S}_0 = \emptyset$, so that $C_1 \ra_{S_0} C_2 = C_1$ and $C_1 \ra_{\overline{S}_0} C_2 = C_2$. Define functions $\varphi_k\colon J^2 \to \mathbb{R}$, for $k = 0, \dots, 3$, as follows:
   \begin{align*}
   \varphi_0(p, \overline{p}) &= \mathbf{1}[p = a|bc],\\
   \varphi_1(p, \overline{p}) &= \mathbf{1}[p = ab|c],\\
   \varphi_2(p, \overline{p}) &= \mathbf{1}[p = ac|b],\\
   \varphi_3(p, \overline{p}) &= -\mathbf{1}[p \in ab \cup ac \text{ and } \overline{p} \in a|b \cap a|c] - \mathbf{1}[p = \overline{p} = ac|b] - \mathbf{1}[p = \overline{p} = ab|c].
   \end{align*}

For each vertex $u \in G$, let $K_u$ be the cluster in $C_1$ containing $u$. Denote by $E_u$ the set of open edges within $K_u$ in $C_1$. Additionally, denote by $\tilde{E}_u$ the set of all edges in $G$ for which at least one endpoint is in $K_u$. By construction, the states of all the edges in $\tilde{E}_c$ are identical in the configurations $C_1$ and $C_1 \ra_{S_1} C_2$; therefore, any event of the form ``there is an open path between $u$ and $c$'', where $u$ is any vertex of $G$, occurs simultaneously in the configurations $C_1$ and $C_1 \ra_{S_1} C_2$. Similarly,

\begin{itemize}
    \item $C_1$ and $C_1 \ra_{S_2} C_2$ coincide on $\tilde{E}_b$; thus, whether an open path between $u$ and $b$ exists is consistent across both $C_1$ and $C_1 \ra_{S_2} C_2$ for any vertex $u$.
    \item $C_1$ and $C_1 \ra_{\overline{S}_3} C_2$ coincide on $\tilde{E}_a$; therefore, the existence of an open path between $u$ and $a$ is the same in both $C_1$ and $C_1 \ra_{\overline{S}_3} C_2$ for any vertex $u$.
\end{itemize}
Finally, we recall an important observation from the proof of Theorem~\ref{3hypernotinbond}: if $C_1 \in a|b|c$ and $C_1 \ra_{S_3} C_2 \in ab \cup ac$, then either $C_1 \ra_{S_1} C_2 \in ab$ or $C_1 \ra_{S_2} C_2 \in ac$.

Now, we are ready to prove that the functions $\varphi_k$ are feasible potentials: the inequality
\[
\sum_{k = 0}^3 \varphi_k(p_k, \overline{p}_k) \ge 0
\]
holds for any feasible combination of partitions $(p_k, \overline{p}_k)$. Since only $\varphi_3$ can take negative values, it is sufficient to consider the following cases where $\varphi_3$ contributes negatively:
\begin{itemize}
    \item $p_3 \in ab \cup ac \text{ and } \overline{p}_3 \in a|b \cap a|c$,
    \item $p_3 = \overline{p}_3 = ab|c$,
    \item $p_3 = \overline{p}_3 = ac|b$.
\end{itemize}
In each of these cases, it is sufficient to demonstrate that either $p_0 = a|bc$, or $p_1 = ab|c$, or $p_2 = ac|b$.

Consider the first case: $p_3 \in ab \cup ac$ and $\overline{p}_3 \in a|b \cap a|c$. Since $C_1$ and $C_1 \ra_{\overline{S}_3} C_2$ coincide on $\tilde{E}_a$, and because the partition of $C_1 \ra_{\overline{S}_3} C_2$ is $\overline{p}_3 \in a|b \cap a|c$, we conclude that $C_1 \in a|b \cap a|c$. This implies that either $C_1 \in a|bc$ or $C_1 \in a|b|c$. 

In the case where $C_1 \in a|bc$, we have $p_0 = a|bc$. Alternatively, if $C_1 \in a|b|c$ and $C_1 \ra_{S_3} C_2 \in ab \cup ac$, then by Proposition~\ref{prop:key} from the proof of Theorem~\ref{3hypernotinbond}, either $C_1 \ra_{S_1} C_2 \in ab$ or $C_1 \ra_{S_2} C_2 \in ac$. This condition translates to either $p_1 \in ab$ or $p_2 \in ac$. Furthermore, since $C_1$ and $C_1 \ra_{S_1} C_2$ coincide on $\tilde{E}_c$ and $C_1 \in c|a \cap c|b$, it follows that $C_1 \ra_{S_1} C_2 \in c|a \cap c|b$ and hence $p_1 \in c|a \cap c|b$. Similarly, $p_2 = b|a \cap b|c$. Thus, $C_1 \in a|b|c$ and $C_1 \ra_{S_3} C_2 \in ab \cup ac$ implies that either $p_1 = ab|c$ or $p_2 = ac|b$, covering all scenarios that ensure that the sum $\varphi_k(p_k, \overline{p}_k)$ is non-negative. This completes the analysis for the first case.

Consider the second case: $p_3 = \overline{p}_3 = ab|c$. Repeating the argument that the configurations $C_1$ and $C_1 \ra_{\overline{S}_3} C_2$ coincide on $\tilde{E}_a$, we find that there is an open path between $a$ and $b$ in $C_1$, and $a$ and $c$ are in different clusters. Thus, $C_1 \in ab|c$, implying that $p_0 = ab|c$ as well.

Next, we observe that if $C_1 \in ab|c$, then the set $S_1$ coincides with $\tilde{E}_c$. This specifically implies that $C_1 \ra_{S_1} C_2 \in c|a \cap c|b$, and we need only to verify that $C_1 \ra_{S_1} C_2 \in ab$.

Consider an open path $\gamma$ between $a$ and $b$ in the configuration $C_1 \ra_{S_3} C_2$. Note that no vertex from this path can belong to $K_c$; otherwise, since the component $K_c$ is connected in $C_1 \ra_{S_3} C_2$, it would imply that $C_1 \ra_{S_3} C_2 \in abc$, which contradicts our assumption. Therefore, all the edges of $\gamma$ must belong to $E \setminus \tilde{E}_c$. Given that $C_1 \in ab|c$, the set $\overline{S}_3$ contains all edges from $E \setminus \tilde{E}_c$; thus, the edge states of $\gamma$ are derived from the configuration $C_2$, ensuring that $\gamma$ is an open path in $C_2$. Additionally, since $\overline{S}_1 = E \setminus \tilde{E}_c$, all edge states of $\gamma$ in $C_1 \ra_{S_1} C_2$ are also from $C_2$, confirming that $\gamma$ is an open path in $C_1 \ra_{S_1} C_2$, which implies $C_1 \ra_{S_1} C_2 \in ab$.

Therefore, $C_1 \ra_{S_1} C_2 \in c|a \cap c|b \cap ab = ab|c$, translating to $p_1 = ab|c$. The third case, where $p_3 = \overline{p}_3 = ac|b$, is fully symmetric, leading to $p_2 = ac|b$.

Altogether, this demonstrates that the functions $\varphi_k$ are feasible potentials. According to Theorem~\ref{thm:dual_potentials}, any feasible distribution $\mu$ on $J$ must satisfy the inequality:
\begin{align*}
\sum_{k = 0}^3 \sum_{(p, \overline{p}) \in J^2} \varphi_k(p, \overline{p})\mu(p)\mu(\overline{p}) = &\ \mu(a|bc) + \mu(ac|b) + \mu(ab|c) \\
& - \mu(ab \cup ac) \cdot \mu(a|b \cap a|c) - \mu(ac|b)^2 - \mu(ab|c)^2 \ge 0.
\end{align*}
\end{proof}

\begin{remark}
    Notice that feasible potentials form a convex cone. An interesting computational task is to enumerate all its extreme rays. We have performed this numerically for the decision trees constructing the sets $S_k$ for $k = 0, \dots, 3$, and their complements, finding exactly three non-trivial rays that form this cone. The first two are responsible for the inequalities \eqref{eq:computer} and \eqref{e2}. The final one leads to the inequality:
    \begin{align*}
    \mu(ab \cup ac) \cdot \mu(a|b \cap a|c) + \mu(a|bc)^2 + \mu(ac|b)^2 + \mu(ab|c)^2 & \\
    \le \mu(a|bc) + \mu(ac|b) + \mu(ab|c) + \mu(abc) \cdot \mu(a|bc), &
    \end{align*}
    which is very similar to \eqref{eq:computer}. The potentials leading to these inequalities can be found on GitHub.\footnote{\href{https://github.com/Kroneckera/bunkbed}{https://github.com/Kroneckera/bunkbed}}
\end{remark}

\section{Further questions}

Inequalitites~\eqref{eq:final} and~\eqref{eq:computer} prove that if all three probabilities $\mu(ab|c)$, $\mu(ac|b)$ and $\mu(a|bc)$ are $0$, then one of $\mu(abc)$ and $\mu(a|b|c)$ should be $0$. In fact, the stronger statement holds: 

\begin{proposition}
If $\mu(ab|c) = 0$, then
$$\mu(a|b|c)\mu(abc) = \mu(ac|b)\mu(a|bc)\text{ and }\mu(abc) = \mu(ac)\mu(bc).$$
\end{proposition}
\begin{proof}
As in the proof of Theorem~\ref{sitenotinbond}, we first delete the edges having probability $0$ and contract the edges having probability $1$. Now all paths from $a$ to $b$ should pass through $c$, otherwise, there will be a nonzero probability of one such path being open and the rest of the edges closed. This means $c$ splits the graph in two parts with $a$ and $b$ belonging to different parts. Thus, the events $ac$ and $bc$ are determined by different sets of edges and consequently are independent.

\end{proof}

However, contrary to the inequalitites~\eqref{eq:final} and~\eqref{eq:computer}, this proof tells nothing when $\mu(ab|c) < \varepsilon$. So, we pose two conjectures increasing in strength:

\begin{conjecture}
For any $\varepsilon > 0$ there exists $\delta > 0$ such that if $\mu(ab|c) < \delta$ and $\mu(ac|b) < \delta$, then $\mu(abc)$ or $\mu(a|b|c)$ is less than $\varepsilon$.
\end{conjecture}

\begin{conjecture}\label{Nikita}
For any $\varepsilon > 0$ there exists $\delta > 0$ such that if $\mu(ab|c) < \delta$, then 
$$\mu(abc) - \mu(ac)\mu(bc) < \varepsilon.$$
\end{conjecture}

\begin{remark}\label{KozmasQ}
    On the contrary, if $\mu(abc) - \mu(ac)\mu(bc) < \varepsilon$, then, by inequality \eqref{e2}, one gets 
    \begin{multline*}
        \mu(ab|c)\Big(\mu(abc) + \mu(ac|b) + \mu(a|bc)\Big) \\ \le \mu(abc)\Big(\mu(abc) + \mu(ab|c) + \mu(ac|b) + \mu(a|bc) + \mu(a|b|c)\Big) \\- \Big(\mu(abc) + \mu(ac|b)\Big)\Big(\mu(abc) + \mu(ab|c)\Big)
    \end{multline*}
    which simplifies to
    $$\mu(ab|c) < \frac{\varepsilon}{\mu(ac \text{ or }bc)}.$$
\end{remark}

And for the final question, finding the exact value for $\alpha_3$ would also be interesting. The best boundaries are given in the Appendix~\ref{alpha3}.

\section{Acknowledgements}
We thank Igor Pak for many helpful comments on the manuscript and Tom Hutchcroft for his thoughtful review and encouragement. We also thank Dmitry Krachun for fruitful discussions and Gady Kozma for the question leading to Remark \ref{KozmasQ}.
Research supported by ERC Advanced Grant GeoScape882971.

\appendix
\section{Appendix: optimizing \texorpdfstring{$\alpha_3$}{alpha_3}}\label{alpha3}

Let us recall that $\alpha_3$ denotes the largest possible value of $\min\big(\mu(abc), \mu(a|b|c)\big)$ for the bond percolation. Let us restrict ourselves to the triangle graph with all three probabilities equal to $p$. Then $\mu(a|b|c) = (1-p)^3$ and $\mu(abc) = p^3 + 3p^2(1-p)$. These numbers coincide for $p \approx 0.3473$, and we get $\alpha_3 \ge \mu(abc) = \mu(a|b|c) \approx 0.278$, a root of the equation $x^3 - 24x^2 + 3x + 1 = 0$.

\begin{figure}[ht]
\centering
\begin{minipage}{.9\textwidth}
\begin{center}
\includegraphics[width=0.2\linewidth]{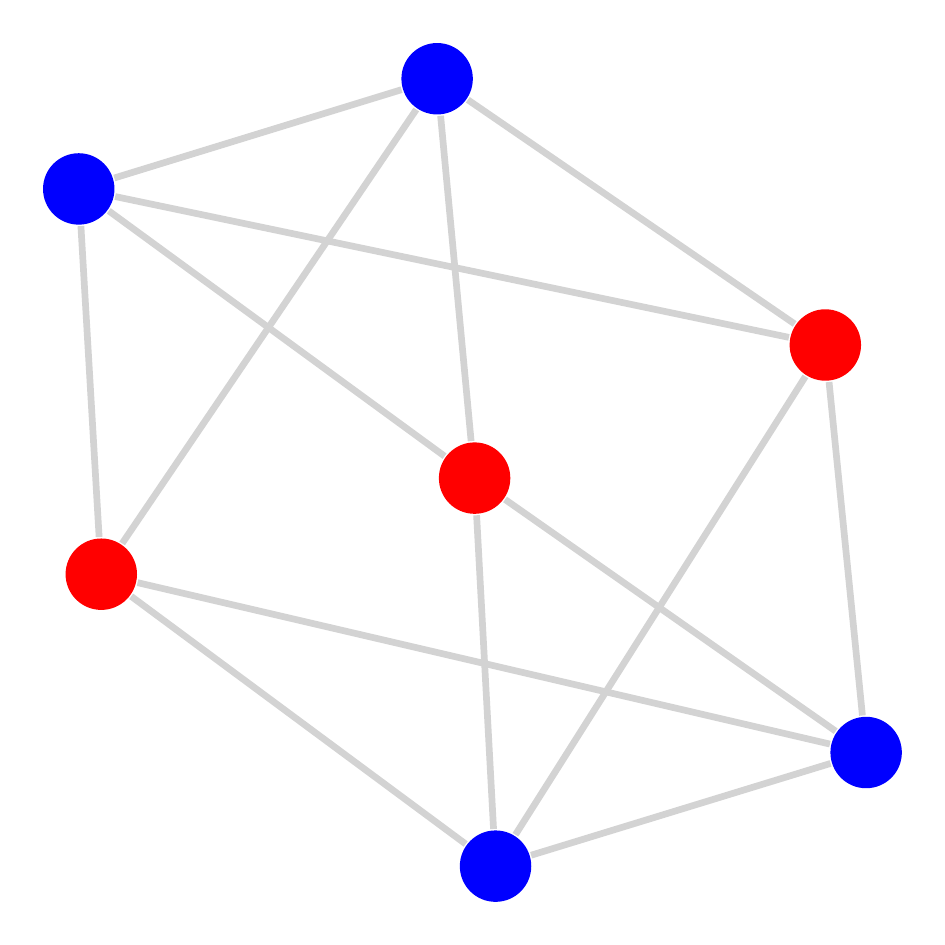}

\caption{Graph for $\alpha_3$.}\label{Fig:K23}
\end{center}
\end{minipage}
\end{figure}

One can do better by utilizing the graph in Figure~\ref{Fig:K23} where each red-blue edge has a probability of $0.32537$ and both blue-blue edges have a probability of $0.19231$. This way we get $\mu(abc) \approx \mu(a|b|c) \approx 0.29065$.

Our computer search using algorithms from Wagner \cite{Wag} wasn't able to beat this estimate (See the best $\min\big(\mu(abc), \mu(a|b|c)\big)$ achieved on each training epoch in Figure~\ref{Fig:Wagner}).  % where $\min(\mu(abc), \mu(a|b|c))$ would be at least $0.3$. 

\begin{figure}[ht]
\centering
\begin{minipage}{.9\textwidth}
\begin{center}
\includegraphics[width=0.5\linewidth]{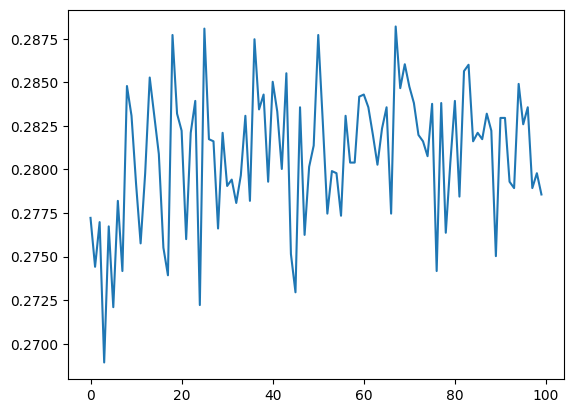}

\caption{Best $\min\big(\mu(abc), \mu(a|b|c)\big)$ achieved on each training epoch.}\label{Fig:Wagner}
\end{center}
\end{minipage}
\end{figure}

In fact, if $\mu(abc) = \mu(a|b|c)$, it seems this probability can only lie in a narrow range from $0.27$ to $0.291$. Indeed, in this case inequality~\eqref{e2} gives the lower bound of $2 - \sqrt{3} \approx 0.2679$.
% END inlined content from Text_v2.tex


\begin{thebibliography}{99}

\bibitem[BK85]{BK}
Jacob van den Berg and Harry Kesten. 
Inequalities with applications to percolation and reliability. \emph{J. Appl. Probab.}~\textbf{22.3} (1985): 556-569.

\bibitem[BHK06]{BHK}
Jacob van den Berg, Olle H\"aggstr\"om, and Jeff Kahn. 
Some conditional correlation inequalities for percolation and related processes. \emph{Random Structures \& Algorithms} \textbf{29.4} (2006): 417-435.

\bibitem[DC18]{DC}
Hugo Duminil--Copin.
Introduction to Bernoulli percolation.
\emph{Lecture notes available on the webpage of the author} (2018), 5~pp.

\bibitem[F61]{Fis}
Michael E. Fisher.
Critical probabilities for cluster size and percolation problems. \emph{J. Mathematical Phys.} \textbf{2.4} (1961): 620-627.

\bibitem[FE61]{FE}
Michael E. Fisher and John W. Essam.
Some cluster size and percolation problems. \emph{J. Mathematical Phys.} \textbf{2.4} (1961): 609-619.

\bibitem[G24]{G}
Nikita Gladkov.
A strong FKG inequality for multiple events,
To appear in \emph{Bull. Lond. Math. Soc}, (2024), 7~pp.

\bibitem[GP24]{GP}
Nikita Gladkov and Igor Pak.
Positive dependence for colored percolation,
\emph{Phys. Rev. E} \textbf{109} (2024), 9~pp.

\bibitem[GPZ54]{GPZ}
Nikita Gladkov, Igor Pak and Aleksandr Zimin.
The bunkbed conjecture is false,
\emph{Proc. Natl. Acad. Sci. USA}~\textbf{122} (2025), 11~pp.

\bibitem[G18]{Gri}
Geoffrey Grimmett.
\emph{Probability on graphs: random processes on graphs and lattices.} Vol. \textbf{8}. Cambridge University Press, (2018).

\bibitem[GS98]{GriSta}
Geoffrey Grimmett and Alan Stacey.
Critical probabilities for site and bond percolation models, \emph{Ann. Probab.} \textbf{26 (4)} (1998), 30~pp.

\bibitem[Hol24]{Hollom}
Lawrence~Hollom.
The bunkbed conjecture is not robust to generalisation,
\emph{European J. Combin.}~\textbf{128} (2025), 16~pp.

\bibitem[H21]{Hut}
Tom Hutchcroft. Power-law bounds for critical long-range percolation below the upper-critical dimension. \emph{Probab. Theory Related Fields}~\textbf{181} (2021): 533-570.

% \bibitem{K22}
% Jeff Kahn,
% A note on positive association.
% arXiv:2210.08653, 2022, 5~pp.

\bibitem[W21]{Wag}
Adam Zsolt Wagner. 
Constructions in combinatorics via neural networks,
arXiv:2104.14516, 2021, 23~pp.

\bibitem[WZ11]{WZ}
John C. Wierman and Robert M. Ziff. 
Self-dual planar hypergraphs and exact bond percolation thresholds. \emph{Electron. J. Combin.}~\textbf{18.1} (2011), 19~pp.


\end{thebibliography}
\end{document}